\documentclass[11pt,a4paper,twoside]{article}
\usepackage{amsmath, amsthm, amscd, amsfonts, amssymb, graphicx, color}
\usepackage[bookmarksnumbered, colorlinks, plainpages]{hyperref}
\usepackage{pdflscape}
\usepackage{blindtext}
\usepackage{float}
\input{mathrsfs.sty}

\newtheorem{theorem}{Theorem}[section]
\newtheorem*{thma}{Theorem A}
\newtheorem*{thmb}{Theorem B}
\newtheorem*{thmc}{Theorem C}

\newtheorem{lemma}[theorem]{Lemma}

\theoremstyle{definition}

\theoremstyle{remark}

\begin{document}
\title{\Large {\bf On harmonic convolutions involving a vertical strip mapping}}
\author{Raj Kumar\,$^a$ \thanks{rajgarg2012@yahoo.co.in}\,\,,\, Sushma Gupta\,$^a$, Sukhjit Singh\,$^a$, and {Michael Dorff\,$^b$\thanks{mdorff@math.byu.edu} }\\
\emph{ \small $^a$\,Sant Longowal Institute of Engineering and Technology, Longowal-148106 (Punjab), India.}\\
\emph{ \small$^b$\, Department of Mathematics, Brigham Young University, Provo, Utah, 84602, USA.}}
\date{}
\maketitle
\begin{abstract}
Let $\displaystyle f_\beta=h_\beta+\overline{g}{_\beta}$ and $F_a=H_a+\overline{G}_a$ be harmonic mappings obtained by shearing of analytic mappings $\displaystyle h_\beta+g_\beta={1}/{(2i{\sin}\beta)}\log\left({(1+ze^{i\beta})}/{(1+ze^{-i\beta})}\right),\,\, 0<\beta<\pi$ and $\displaystyle H_a+G_a={z}/{(1-z)}$, respectively. Kumar et al. [\ref{ku and gu}] conjectured that if $\omega(z)=e^{i\theta}z^n (\theta\in\mathbb{R}\,\, n\in \mathbb{N})$ and $\displaystyle \omega_a(z)={(a-z)}/{(1-az)},\,a\in(-1,1)$ are dilatations of $f_\beta$ and $F_a$, respectively, then $F_a\ast f_\beta \, \in S_H^0$ and is convex in the direction of the real axis provided $a\in \displaystyle \left[{(n-2)}/{(n+2)},1\right)$. They claimed to have verified the result for $n=1,2,3$ and $4$ only. In the present paper, we settle the above conjecture in the affirmative for $\beta=\pi/2$ and for all $n\in \mathbb{N}$.  \\

\end{abstract}

\vspace{1mm}
{\small

}

\section{Introduction}
Let $\mathcal{H}$ be the class of all complex valued harmonic functions $f$ defined in the unit disk $E = \{z: |z|<1\}$ and normalized by the conditions $f(0)=0$ and $f_z(0)=1$. Such harmonic mappings can be decomposed as $f=h+\overline g$, where $h$ is known as the analytic and $g$ the co-analytic part of $f$. A harmonic mapping $f=h+\overline g$ defined in $E$, is locally univalent and sense-preserving if and only if $h'\not=0$ in $E$ and the dilatation function $\omega,$ defined by $\displaystyle\omega={g'}/{h'}$, satisfies $|\omega(z)|<1$ for all $z\,\in E$. We denote by $S_H$ the class of all univalent and sense-preserving harmonic mappings in $\mathcal{H}$. Functions $f=h+\overline g$ in the class $S_H$ has the representation
\begin{equation} f(z) = z+ \sum _{n=2}^{\infty} a_nz^n + \sum _{n=1}^{\infty}\overline{ b}{_n}\overline{z}^n \,\,\,{\rm  for\,\, all}\,\, z\,\,in\,\, E. \end{equation}  The class of functions of the type (1) with $b_1=0$ is a subset of $S_H$ and will be denoted by $S_H^0$ here.
Further let $K_H$(respectively $K_H^0$) be the subclass of  $S_H$(respectively $S_H^0$) consisting of functions which map the unit disk $E$ onto convex domains. A domain $\Omega$ is said to be convex in the direction $\phi,\, 0 \leq \phi < \pi,$ if every line parallel to the line through $0$ and $ e ^{i \phi}$ has either connected or empty intersection with $\Omega$. In particular a domain convex in the direction of the real axis is denoted by CHD.\\
Let $\displaystyle f_\beta=h_\beta+\overline{g}{_\beta}$ be the collection of those harmonic mappings which are obtained by shearing (see Clunie and Sheil-Small [\ref{cl and sh}]) of analytic vertical strip mappings, \begin{equation}\displaystyle h_\beta+g_\beta=\frac{1}{2i{\sin}\beta}\log\left(\frac{1+ze^{i\beta}}{1+ze^{-i\beta}}\right),\end{equation} where $\displaystyle 0<\beta<\pi.$\\
Harmonic Convolution or simply convolution of two functions $F(z) = H(z) + \overline G(z)=z + \sum_{n=2}^\infty A_n z^n  +
   \sum_{n=1}^\infty{\overline B}{_n} {\overline z} {^n}$ and
$ f(z)= h(z)+ \overline g(z)  =z+\sum_{n=2}^\infty a_n z^n  +   \sum_{n=1}^\infty\overline
   {b}{_n}\overline{z}{^n}$ in $S_H$ is defined as,
$$ \begin{array}{clll}
(F {\ast} f)(z)&=&  (H {\ast} h)(z) +\overline{(G {\ast} g)}(z)\\
 &= & z+ \sum_{n=2}^\infty a_n A_n z^n + \sum_{n=1}^\infty\overline {{b}}{_n}\overline{{B}}{_n}\overline {z}{^n}.
\end{array} $$

The behaviour of the harmonic convolutions is not as nice as that of the analytic convolutions. Unlike the case of anlytic mappings, convolution of two convex harmonic mappings is not necessarily convex harmonic. But still convolutions of harmonic functions exhibit some interesting properties. For example Dorff [\ref{do}] proved that convolution of two harmonic right half-plane mappings is CHD. In the same paper he investigated convolution of mapping $f_{\beta}$ with right half-plane mapping and obtained the following:
\begin{thma} (See Dorff [\ref{do}]) Let $\displaystyle f=h+\overline{g}\in K_H^0$ with $\displaystyle h+g={z}/{(1-z)}$. Then $f\ast f_\beta\,\in S_H^0$ and is CHD, provided $f\ast f_\beta$ is locally univalent and sense-preserving, where $f_\beta$ is given by (2) and $\displaystyle \beta \in [\pi/2,\pi)$.\end{thma}
Later, Dorff et al. [\ref{do and no}] proved that the condition of $f\ast f_\beta$ being locally univalent and sense-preserving can be dropped in some special cases. They proved:
\begin{thmb} (See Dorff et al. [\ref{do and no}])  Let $\displaystyle f_\beta$ be as in Theorem A with dialatation $\omega(z)=e^{i\theta}z^n,\,\theta\in \mathbb{R}$. If $F_0=H_0+\overline{G}_0$ where $\displaystyle H_0+G_0={z}/{(1-z)}$ with dilatation $\omega_0(z)=-z$, then $F_0\ast f_\beta\,\in S_H^0$ and is CHD for $n=1,2$.\end{thmb}
Kumar et al. [\ref{ku and do}], defined harmonic right half-plane mappings $F_a=H_a+\overline{G}_a,$ where \begin{equation} H_a+G_a=\frac{z}{1-z}\,\,\text{with
\,dilatations} \,\,\omega_a(z)=\frac{a-z}{1-az},\,a\in(-1,1).\end{equation}
Note that for $a=0$ the mapping $F_a$ reduces to the mapping $F_0$ of Theorem B.\\
Recently, in [\ref{ku and gu}] Kumar et al. investigated the harmonic convolution of mapping $f_\beta$ with mapping $F_a$ and essentialy proved the following result, which generalizes Theorem B:
\begin{thmc} If $f_\beta= h_\beta+\overline{g}{_\beta}$ are the harmonic mappings given by (2) with dilatations $\omega(z)=e^{i\theta}z^n, \theta\in\mathbb{R}$, then $F_a\ast f_\beta \, \in S_H^0$ and are CHD for $n=1,2,3$ provided $a$ is restricted in the interval $\displaystyle  \left[{(n-2)}/{(n+2)},1\right).$\end{thmc}
At the end of the above paper authors remarked that Theorem C continues to hold true for $n=4$ and conjectured the following general result.

{\bf Conjecture.} Let $f_\beta= h_\beta+\overline{g}{_\beta}$ be as in Theorem C. Then $F_a\ast f_\beta \, \in S_H^0$ and is CHD for all $n\in \mathbb{N}$ provided $\displaystyle a \in \left[{(n-2)}/{(n+2)},1\right)$.\\
\indent In the present paper, we settle this conjecture in the affirmative for $\beta=\pi/2$.

\section{Main Results}

The following results will be used in proving our main theorem.

\begin{lemma} (Cohn's Rule [\ref{ra and sc}, p.375])  Given a polynomial $$t(z)= a_0 + a_1z + a_2z^2+...+a_nz^n$$ of
degree $n$, let $$ t^*(z)=\displaystyle z^n\overline{t\left(\frac{1}{\overline z}\right)} = \overline {a}_n + \overline {a}_{n-1}z + \overline {a}_{n-2}z^2 +...+ \overline{a}_0z^n.$$
Denote by $r$ and $s$ the number of zeros of $t$ inside and on the unit circle $|z|=1$, respectively.
 If $|a_0|<|a_n|,$ then $$ t_1(z)= \frac{\overline {a}_n t(z)-a_0t^*(z)}{z}$$ is of degree $n-1$ and has $r_1=r-1$ and $s_1=s$ number
 of zeros inside and on the unit circle $|z|=1$, respectively.\end{lemma}

\begin{lemma} (Schur-Cohn algorithm [\ref{ra and sc}, p.383])  Given a polynomial
 $$ r(z) =a_0+a_1z+\cdots+a_{n}z^{n}$$ of degree $n$, let \[M_{k}= det\begin{pmatrix}
\overline{B}_{k}\,^T& A_{k} \\
\overline{A}_{k}\,^T& B_{k}
\end{pmatrix} \begin{pmatrix}
k=1,2\cdots,n
\end{pmatrix},\] \\ where $A_{k}$ and $B_{k}$ are the triangular matrices

\[A_{k}=\begin{pmatrix}
a_{0}&a_{1}&\cdots &a_{k-1} \\
&a_{0}&\cdots & a_{k-2}\\
 &  &  \ddots & \vdots\\
& &  &a_{0}
\end{pmatrix},\qquad\,\,\,
 B_{k}=\begin{pmatrix}
\overline{a}_{n}&\overline{a}_{n-1}&\cdots &\overline{a}_{n-k+1} \\
&\overline{a}_{n}&\cdots &\overline{a}_{n-k+2} \\

 & & \ddots & \vdots\\
& &  &\overline{a}_{n}
\end{pmatrix}.\]\\
\normalsize{
Then the polynomial $r$ has all its zeros inside the unit circle $|z|=1$ if and only if the determinants $M_1, M_2\cdots,M_{n}$ are all positive.}\end{lemma}
\begin{lemma}(See Kumar et al. [\ref{ku and gu}]) If $f_\beta= h_\beta+\overline{g}_{\beta}$ is the mapping given by (2) with dilatation $\displaystyle\omega={g'_\beta}/{h'_\beta}$, then $\widetilde{\omega}$, the dilatation of $F_a\ast f_\beta,$ is given by \begin{equation} \displaystyle \widetilde{\omega}(z)=\left[\frac{2\omega(1+\omega)(a+az\cos\beta+z\cos\beta+z^2)-z\omega'(1-a)(1+2z\cos\beta+z^2)}{2(1+z\cos\beta+az\cos\beta+az^2)(1+\omega)-z\omega'(1-a)(1+2z\cos\beta+z^2)}\right].
\end{equation}
\end{lemma}

We now proceed to state and prove our main result.

\begin{theorem} Let $F_a=H_a+\overline{G}_{a}$ be given by (3) and  $\displaystyle f_{{\pi}/{2}}= h_{{\pi}/{2}}+\overline{g}_{{{\pi}/{2}}}$ be the map, obtained from (2) with $\displaystyle \beta = {\pi}/{2}$ and dilatation $\displaystyle \omega(z)={g'_{{\pi}/{2}}(z)}/{h'_{{\pi}/{2}}(z)}=e^{i\theta}z^n\,\,(\theta \in \mathbb{R}\,, n\in\mathbb{N})$. Then $ \displaystyle F_a\ast f_{{\pi}/{2}} \, \in S_H^0$ and is CHD for $\displaystyle a \in \left[{(n-2)}/{(n+2)},1\right)$.\end{theorem}
\begin{proof} In view of Theorem A, we need only to show that $F_a\ast f_{{\pi}/{2}}$ is locally univalent and sense-preserving or equivalently the dilatation $\widetilde{\omega}\,$ of $F_a\ast f_{{\pi}/{2}}$ satisfies $\displaystyle|\widetilde{\omega}(z)|<1$ for all $z \in E$. Setting $\omega(z)=e^{i\theta}z^n$ and $\displaystyle \beta ={\pi}/{2}$ in (4), we get
\begin{equation}
\indent\hspace{-.7cm}\displaystyle\widetilde{\omega}(z)=\displaystyle z^ne^{2i\theta}\left[\frac{z^{n+2}+az^n+ \frac{1}{2}(2+an-n)e^{-i\theta}z^2+\frac{1}{2}(2a+an-n)e^{-i\theta}}{\frac{1}{2}(2a+an-n)e^{i\theta}z^{n+2}+\frac{1}{2}(2+an-n)e^{i\theta}z^n+az^2+1}\right]
\end{equation}

$\indent\hspace{1.3cm}=\displaystyle z^ne^{2i\theta}\frac{p(z)}{p^*(z)},$\\
\noindent where
\begin{equation}\hspace{-1.8cm}p(z)=z^{n+2}+az^n+ \frac{1}{2}(2+an-n)e^{-i\theta}z^2+\frac{1}{2}(2a+an-n)e^{-i\theta}
\end{equation}
and \qquad $p^*(z)=z^{n+2}\overline{p\left(\frac{1}{\overline z}\right)}.$\\
\noindent Obviously, if $z_0$, $z_0\not=0,$ is a zero of $p$ then $\displaystyle\frac{1}{\overline{z}_0}$ is a zero of $p^*$. Hence, if $\alpha_1,\alpha_2\cdots \alpha_{n+2}$ are the zeros of $p$ (not necessarily distinct), then we can write
$$\displaystyle\widetilde{\omega}(z)=z^ne^{2i\theta}\frac{(z-\alpha_1)}{(1-\overline {\alpha}_1 z)}\frac{(z-\alpha_2)}{(1-\overline {\alpha}_2 z)}\cdots\frac{(z-\alpha_{n+2})}{(1-\overline {\alpha}_{n+2} z)}.$$
 For $|\alpha_i|\leq1$, since $\displaystyle \frac{(z-\alpha_i)}{(1-\overline {\alpha}_{i} z)}$ maps the closed unit disk onto itself, therefore to prove that $\displaystyle|\widetilde{\omega}|<1$ in $E,$ we shall show that all the zeros of polynomial $p$ i.e, $\alpha_1,\alpha_2,\cdots,\alpha_{n+2}$ lie inside or on the unit circle $|z|=1$ for $a\in \left({(n-2)}/{(n+2)},1\right)-\{n/(n+2)\}$ (as in the case when $a={(n-2)}/{(n+2)},$ we see from (5) that $\displaystyle|\widetilde{\omega}(z)|=\displaystyle|-z^ne^{i\theta}|<1$ and the case when $a=n/(n+2),$ where $z_0=0$ is a zero of $p,$ will be taken up separately). To do this we use Lemma 2.2 on the polynomial $ p(z)=z^{n+2}+az^n+ \frac{1}{2}\left(2+an-n\right)e^{-i\theta}z^2+\frac{1}{2}\left(2a+an-n\right)e^{-i\theta}$ and show that all the deteminants $M_k\,\, (k=1,2,3,\cdots,n+2)$ are positive, where $M_k$ is obtained by comparing the polynomial $p$ with polynomial $r$ (with degree $n+2$) of Lemma 2.2 and is given by
\[\indent\hspace{-6cm}M_{k}= det\begin{pmatrix}
\overline{B}_{k}\,^T& A_{k} \\
\overline{A}_{k}\,^T& B_{k}
\end{pmatrix}\,\,(k=1,2,3,\cdots,n+2).\] For entries of $A_k$ and $B_k$ we obviously have $a_{0}=\frac{1}{2}(2a+an-n)e^{-i\theta}, a_{1}=0, a_{2}=\frac{1}{2}(2+an-n)e^{-i\theta}, a_{3}=0, \cdots a_{n-1}=0,$ $a_{n}=a, a_{n+1}=0,$ and $a_{n+2}=1.$\\
Since $B_{k}$ is a non singular matrix, therefore,
\[\begin{pmatrix}
\overline{B}_{k}\,^T& A_{k} \\
\overline{A}_{k}\,^T& B_{k}
\end{pmatrix}\begin{pmatrix}
I& \Huge{0} \\
-B_{k}^{-1}\overline{A}_{k}\,^T& I
\end{pmatrix}=\begin{pmatrix}
\overline{B}_{k}\,^T- A_{k}B_{k}^{-1}\overline{A}_{k}\,^T& A_{k} \\
0 & B_{k}
\end{pmatrix}\] which gives
 \[\indent\hspace{-1cm}M_{k}= det\begin{pmatrix}
\overline{B}_{k}\,^T& A_{k} \\
\overline{A}_{k}\,^T& B_{k}
\end{pmatrix}=det\begin{bmatrix}
\overline{B}_{k}\,^T- A_{k}B_{k}^{-1}\overline{A}_{k}\,^T\end{bmatrix}.\]
 For $n=1,2,3$ and $4,$ as the proof of our theorem follows from the results in [\ref{ku and gu}] by substituting $\beta=\pi/2$, so we take $n\geq5.$ We consider the following cases.\\

{\bf Case 1.} \emph{ When $1 \leq\, k\, \leq\, n$ and $n\geq 5$ is odd}. In this case $A_k$ and $B_k$ are the following $k \times k$ matrices;
\scriptsize{
\[\indent\hspace{-2.5cm}A_{k}=\begin{pmatrix}
a_{0}&a_{1}&a_2 &\cdots &a_{k-1} \\
0&a_{0}&a_1&\cdots &a_{k-2} \\
0& 0&a_0&\cdots &a_{k-3}\\
\vdots & \vdots &\vdots& \ddots & \vdots\\
0&0&0&\cdots &a_{0}
\end{pmatrix}=\begin{pmatrix}
a_{0}&0&a_2 &\cdots &0 \\
0&a_{0}&0&\cdots &0 \\
0& 0&a_{0}&\cdots &0\\
\vdots & \vdots &\vdots& \ddots & \vdots\\
0&0&0&\cdots &a_{0}
\end{pmatrix}\quad \text{and}\quad B_{k}=\begin{pmatrix}
\overline{a}_{n+2}&\overline{a}_{n+1}&\overline{a}_{n} &\cdots &\overline {a}_{(n+2)-k+1} \\
0&\overline{a}_{n+2}&\overline{a}_{n+1}&\cdots &\overline{a}_{(n+2)-k+2}  \\
0& 0&\overline{a}_{n+2}&\cdots &\overline{a}_{(n+2)-k+3}\\
\vdots & \vdots &\vdots& \ddots & \vdots\\
0&0&0&\cdots &\overline{a}_{n+2}
\end{pmatrix}=\begin{pmatrix}
{1}&{0}&a &\cdots &0 \\
0&1&0&\cdots &0 \\
0& 0&1&\cdots &0\\
\vdots & \vdots &\vdots& \ddots & \vdots\\
0&0&0&\cdots &1
\end{pmatrix}.\]\\
\normalsize
So,
\scriptsize{
\[ \indent\hspace{-4cm}\overline{B}_{k}^{\,\,T}-A_{k}B_{k}^{-1}\overline{A}_{k}^{\,\,T}=\left( \begin{array}{llll}
1-a_0\overline{a}_0-\overline{a}_2(-aa_0+a_2)& \quad0 &-\overline{a}_0(-aa_0+a_2)+a\overline{a}_2(-aa_0+a_2)&\cdots\\
\quad0& 1-\overline{a}_0a_0-\overline{a}_2(-aa_0+a_2) &\quad0&\cdots \\
a-a_0\overline{a}_2&\quad0 &1-a_0\overline{a}_0-\overline{a}_2(-aa_0+a_2)\\
\quad\vdots &\quad\vdots &\quad\vdots&\ddots \\
\quad0&\quad0&\quad0&\cdots\\
\quad0&\quad0&\quad0&\cdots\\
\quad0&\quad0&\quad0&\cdots
\end{array} \right.\]}
\[\indent\hspace{2.5cm}\left.\begin{array}{clll}
(-a)^{\frac{k-5}{2}}[-\overline{a}_0(-aa_0+a_2)+a\overline{a}_2(-aa_0+a_2)]&\quad 0 &(-a)^{\frac{k-3}{2}}[-\overline{a}_0(-aa_0+a_2)] \\
\quad0 & (-a)^{\frac{k-5}{2}}[-\overline{a}_0(-aa_0+a_2)] &\qquad0\\
(-a)^{\frac{k-7}{2}}[-\overline{a}_0(-aa_0+a_2)+a\overline{a}_2(-aa_0+a_2)]&\quad 0&(-a)^{\frac{k-5}{2}}[-\overline{a}_0(-aa_0+a_2)]\\
\vdots&\quad\vdots &\qquad \vdots\\
1-a_0\overline{a}_0-\overline{a}_2(-aa_0+a_2)& \quad0&[-\overline{a}_0(-aa_0+a_2)]\\
0&1-a_0\overline{a}_0&\qquad0\\
a-a_0\overline{a}_2&\quad0 & 1-a_0\overline{a}_0
\end{array} \right).\]

\normalsize Now\\
\noindent{\bf (a)}  $1-\overline{a}_0a_0-\overline{a}_2(-aa_0+a_2)=\frac{1}{4}n(2-n+2a+an)(1-a)(2-a).$\\
${\bf (b)} -\overline{a}_0(-aa_0+a_2)+a\overline{a}_2(-aa_0+a_2)=\frac{1}{4}n(2-n+2a+an)(1-a)^3.$\\
\noindent{\bf (c)} $-\overline{a}_0(-aa_0+a_2)=\frac{1}{4}(n-2a-an)(2-n+2a+an)(1-a).$\\
\noindent{\bf (d)} $a-a_0\overline{a}_2=\frac{1}{4}n(2-n+2a+an)(1-a).\\$
\noindent{\bf (e)} $ 1-{a_0}\overline{a}_0=\frac{1}{4}(n+2)(2-n+2a+an)(1-a).$\\

If we write $\displaystyle L_r=\left(\frac{1}{4}\right)^r n^{r-2}(2-n+2a+an)^r(1-a)^r, r\in \mathbb{N},$ then
\scriptsize{
\[\indent\hspace{-1cm}M_k=L_k\,\,det\begin{pmatrix}
(2-a)&0&(1-a)^2 &\cdots&0 &(-a)^{\frac{k-3}{2}}(n-2a-an)\\
0&(2-a)&0&\cdots&(-a)^{\frac{k-5}{2}}(n-2a-an) &0 \\
1& 0&(2-a)&\cdots&0 &(-a)^{\frac{k-5}{2}}(n-2a-an)\\
\vdots & \vdots &\vdots& \ddots & \vdots & \vdots\\
0&0&0&\cdots&(n+2) &0\\
0&0&0&\cdots&0 &(n+2)
\end{pmatrix}\]}
\scriptsize{
\[\indent\hspace{-.5cm}= L_k\,\, det\begin{pmatrix}
(2-a)&0&(1-a)^2 &\cdots&0 &(-a)^{\frac{k-3}{2}}(n-2a-an)\\
0&(2-a)&0&\cdots&(-a)^{\frac{k-5}{2}}(n-2a-an) &0 \\
0& 0&\frac{3-2a}{2-a}&\cdots&0 &2(-a)^{\frac{k-5}{2}}\frac{(n-2a-an)}{2-a}\\
\vdots & \vdots &\vdots& \ddots & \vdots& \vdots\\
0&0&0&\cdots &\frac{n+k-1}{\frac{k-1}{2}-\frac{(k-3)a}{2}}&\quad 0\\
0&0&0&\cdots &0&\frac{n+k+1}{\frac{k+1}{2}-\frac{(k-1)a}{2}}
\end{pmatrix}\]}
\normalsize
$\indent\hspace{1cm}=L_k(n+k-1)(n+k+1).$\\
As $L_k$ is positive for $a \in \left({(n-2)}/{(n+2)},1\right)-\{n/(n+2)\},$ so $M_k$ is positive in this case.\\

\normalsize
\indent\hspace{-.7cm}{\bf Case 2.} \emph{ When $1 \leq\, k\, \leq\, n$ and $n\geq 5$ is even}. In this case\\
\scriptsize{
\[\indent\hspace{-3cm}A_{k}=\begin{pmatrix}
a_{0}&a_{1}&a_{2} &\cdots &{a_{k-1}} \\
0&a_{0}&a_1&\cdots &a_{k-2} \\
0& 0&a_0&\cdots &a_{k-3}\\
\vdots & \vdots &\vdots& \ddots & \vdots\\
0&0&0&\cdots &a_{0}
\end{pmatrix}=\begin{pmatrix}
a_{0}&0&a_{2} &\cdots &0 \\
0&a_{0}&0&\cdots &0 \\
0& 0&a_0&\cdots &0\\
\vdots & \vdots &\vdots& \ddots & \vdots\\
0&0&0&\cdots &a_{0}
\end{pmatrix}, \quad B_{k}=\begin{pmatrix}
\overline{a}_{n+2}&\overline{a}_{n+1}&\overline {a}_{n} &\cdots &\overline{ a}_{(n+2)-k+1} \\
0&\overline{a}_{n+2}&\overline{a}_{n+1}&\cdots &\overline{ a}_{(n+2)-k+2} \\
0& 0&\overline{a}_{n+2}&\cdots &\overline{ a}_{(n+2)-k+3}\\
\vdots & \vdots &\vdots& \ddots & \vdots\\
0&0&0&\cdots &\overline{a}_{n+2}
\end{pmatrix}=\begin{pmatrix}
1&0&a &\cdots &0 \\
0&1&0&\cdots &0 \\
0& 0&1&\cdots &0\\
\vdots & \vdots &\vdots& \ddots & \vdots\\
0&0&0&\cdots &1
\end{pmatrix}\]}
\normalsize
and so,
 \scriptsize{
\[ \indent\hspace{-4cm}\overline{B}_{k}^{\,\,T}-A_{k}B_{k}^{-1}\overline{A}_{k}^{\,\,T}=\left( \begin{array}{llll}
1-a_0\overline{a}_0-\overline{a}_2(-aa_0+a_2)& \quad0 &-\overline{a}_0(-aa_0+a_2)+a\overline{a}_2(-aa_0+a_2)&\cdots\\
\quad0& 1-\overline{a}_0a_0-\overline{a}_2(-aa_0+a_2) &\quad0&\cdots \\
a-a_0\overline{a}_2&\quad0 &1-a_0\overline{a}_0-\overline{a}_2(-aa_0+a_2)\\
\quad\vdots &\quad\vdots &\quad\vdots&\ddots \\
\quad0&\quad0&\quad0&\cdots\\
\quad0&\quad0&\quad0&\cdots\\
\quad0&\quad0&\quad0&\cdots
\end{array} \right.\]}
\[\indent\hspace{2.5cm}\left.\begin{array}{clll}
\qquad0&(-a)^{\frac{k-4}{2}}[-\overline{a}_0(-aa_0+a_2)] &\qquad 0\\
(-a)^{\frac{k-6}{2}}[-\overline{a}_0(-aa_0+a_2)+a\overline{a}_2(-aa_0+a_2)] & \qquad 0 &(-a)^{\frac{k-4}{2}}[-\overline{a}_0(-aa_0+a_2)] \\
\qquad 0&(-a)^{\frac{k-6}{2}}[-\overline{a}_0(-aa_0+a_2)]&\qquad 0\\
\qquad\vdots&\qquad\vdots &\qquad \vdots\\
1-a_0\overline{a}_0-\overline{a}_2(-aa_0+a_2)& \qquad0&[-\overline{a}_0(-aa_0+a_2)]\\
\qquad0&1-a_0\overline{a}_0& \qquad0\\
a-a_0\overline{a}_2&\qquad0 & 1-a_0\overline{a}_0
\end{array} \right).\]\\
\normalsize
Therefore,
\scriptsize{
\[\indent\hspace{-2cm}M_k= L_K\,\,det\begin{pmatrix}
(2-a)&0&(1-a)^2 &\cdots&(-a)^{\frac{k-4}{2}}(n-2a-an) &\qquad 0\\
0&(2-a)&0&\cdots&\qquad 0 &(-a)^{\frac{k-4}{2}}(n-2a-an) \\
1& 0&(2-a)&\cdots&(-a)^{\frac{k-6}{2}}(n-2a-an) &\qquad 0\\
\vdots & \vdots &\vdots& \ddots& \qquad\vdots &\qquad \vdots\\
0&0&0&\cdots&(n+2) &\qquad0\\
0&0&0&\cdots&0 &(n+2)
\end{pmatrix}\]}
\scriptsize{
\[\indent\hspace{-1.5cm}= L_K\,\,det\begin{pmatrix}
(2-a)&0&(1-a)^2 &\cdots&(-a)^{\frac{k-4}{2}}(n-2a-an) &\qquad 0\\
0&(2-a)&0&\cdots&\qquad 0 &(-a)^{\frac{k-4}{2}}(n-2a-an) \\
0& 0&\frac{3-2a}{2-a}&\cdots&2(-a)^{\frac{k-6}{2}}\frac{(n-2a-an)}{2-a} &\qquad 0\\
\vdots & \vdots &\vdots& \ddots & \qquad\vdots&\qquad \vdots\\
0&0&0&\cdots &\frac{n+k}{\frac{(k-1)+1}{2}-\frac{((k-1)-1)a}{2}}&\qquad 0\\
0&0&0&\cdots &\qquad 0&\frac{n+k}{\frac{k}{2}-\frac{(k-2)a}{2}}
\end{pmatrix}\]}
\normalsize
$\indent\hspace{.5cm}=L_k(n+k)(n+k),$ which is positive for $a \in \left({(n-2)}/{(n+2)},1\right)-\{n/(n+2)\}.$\\

\noindent{\bf Case 3.} \emph{When $k=n+1$ and $n$ is an odd positive integer}. We have,
\scriptsize{
\[\indent\hspace{-3cm}A_{n+1}=\begin{pmatrix}
a_{0}&a_{1}&a_{2} &\cdots &{a_n} \\
0&a_{0}&a_1&\cdots &a_{n-1} \\
0& 0&a_0&\cdots &a_{n-2}\\
\vdots & \vdots &\vdots& \ddots & \vdots\\
0&0&0&\cdots &a_{0}
\end{pmatrix}=\begin{pmatrix}
a_{0}&0&a_{2} &\cdots &a \\
0&a_{0}&0&\cdots &0 \\
0& 0&a_0&\cdots &0\\
\vdots & \vdots &\vdots& \ddots & \vdots\\
0&0&0&\cdots &a_{0}
\end{pmatrix}\quad and \quad B_{n+1}=\begin{pmatrix}
\overline{a}_{n+2}&\overline{a}_{n+1}&\overline {a}_{n} &\cdots &\overline{ a}_2 \\
0&\overline{a}_{n+2}&\overline{a}_{n+1}&\cdots &\overline{ a}_3 \\
0& 0&\overline{a}_{n+2}&\cdots &\overline{ a}_4\\
\vdots & \vdots &\vdots& \ddots & \vdots\\
0&0&0&\cdots &\overline{a}_{n+2}
\end{pmatrix}=\begin{pmatrix}
1&0&a &\cdots &\overline{a}_2 \\
0&1&0&\cdots &0 \\
0& 0&1&\cdots &0\\
\vdots & \vdots &\vdots& \ddots & \vdots\\
0&0&0&\cdots &1
\end{pmatrix}.\]}
\normalsize
 Thus
\scriptsize{
\[ \indent\hspace{-3.3cm}\overline{B}_{n+1}^{\,\,T}-A_{n+1}B_{n+1}^{-1}\overline{A}_{n+1}^{\,\,T}=\left( \begin{array}{llll}
1-a_0\overline{a}_0-\overline{a}_2(-aa_0+a_2)-a(-a_0\overline{a}_2+a)&\qquad 0&-\overline{a}_0(-aa_0+a_2)+a\overline{a}_2(-aa_0+a_2)&\quad\cdots\\
(-a)^{\frac{n-1}{2}}(-aa_0+a_2)&1-a_0\overline{a}_0-\overline{a}_2(-aa_0+a_2)&\quad 0 &\quad\cdots \\
a-a_0\overline{a}_2&\qquad 0 &1-a_0\overline{a}_0-\overline{a}_2(-aa_0+a_2)&\quad\cdots\\
\qquad\vdots &\qquad\vdots &\quad\vdots &\quad \ddots \\
-a(-aa_0+a_2)&\qquad0&\quad0&\quad\cdots\\
\qquad0&\qquad0&\quad0&\quad\cdots\\
(-aa_0+a_2)&\qquad0&\quad0&\quad\cdots
\end{array} \right.\]}
\[\indent\hspace{3cm}\left.\begin{array}{clll}
-\overline{a}_2(-a_0\overline{a}_2+a)&(-a)^{\frac{n-3}{2}}[-\overline{a}_0(-aa_0+a_2)] &-\overline{a}_0(-a_0\overline{a}_2+a) \\
a^{\frac{n-5}{2}}[-\overline{a}_0(-aa_0+a_2)+a\overline{a}_2(-aa_0+a_2)] &\qquad 0&(-a)^{\frac{n-3}{2}}[-\overline{a}_0(-aa_0+a_2)]\\
0&(-a)^{\frac{n-5}{2}}[-\overline{a}_0(-aa_0+a_2)]&\qquad 0\\
\vdots &\qquad\vdots &\qquad \vdots\\
1-a_0\overline{a}_0-\overline{a}_2(-aa_0+a_2)&\qquad0&[-\overline{a}_0(-aa_0+a_2)]\\
0 & 1-a_0\overline{a}_0&\qquad 0\\
a-a_0\overline{a}_2 &\qquad 0 &  1-a_0\overline{a}_0
\end{array} \right).\]\\
\normalsize If $E_j$ is the $j^{th}$ column of $\overline{B}_{n+1}^{\,\,T}-A_{n+1}B_{n+1}^{-1}\overline{A}_{n+1}^{\,\,T},$ then the entries of $E_j\,\, (j=2,3,\cdots,n-2)$ and $E_n$ are identical to those of the corresponding columns of $\overline{B}_{k}^{\,\,T}-A_{k}B_{k}^{-1}\overline{A}_{k}^{\,\,T}$ in Case 2. However, the entries for $E_1$, $E_{n-1}$, and $E_{n+1}$ are different. We split $E_1$, $E_{n-1}$ and $E_{n+1}$ in the following way:\\
$E_1=F_1+G_1+H_1$,  $E_{n-1}=F_{n-1}+G_{n-1}$, $E_{n+1}=F_{n+1}+G_{n+1},$ where\\
$F_1^{T}=(1-a_0\overline{a}_0-\overline{a}_2(-aa_0+a_2),0,a-a_0\overline{a}_2,0,\cdots,0),$\\
$G_1^{T}=(0,(-a)^{\frac{n-1}{2}}(-aa_0+a_2),0,(-a)^{\frac{n-3}{2}}(-aa_0+a_2),0\cdots,(-aa_0+a_2)),$\\
\vspace{.3cm}
$H_1^{T}=(-a(-a_0\overline{a}_2+a),0,0,\cdots,0);$\\
$F_{n-1}^{T}=(0,(-a)^{\frac{n-5}{2}}[-\overline{a}_0(-aa_0+a_2)+a\overline{a}_2(-aa_0+a_2)],0,(-a)^{\frac{n-7}{2}}[-\overline{a}_0(-aa_0+a_2)+\\\indent\hspace{1cm}a\overline{a}_2(-aa_0+a_2)],0\cdots, a-a_0\overline{a}_2),$\\
\vspace{.3cm}
$G_{n-1}^{T}=(-\overline{a}_2(-a_0\overline{a}_2+a),0,0,\cdots,0);$\\
$F_{n+1}^{T}=(0,(-a)^{\frac{n-3}{2}}[-\overline{a}_0(-aa_0+a_2)],0,(-a)^{\frac{n-3}{2}}[-\overline{a}_0(-aa_0+a_2)],0,\cdots,1-a_0\overline{a}_0)$ and\\
$G_{n+1}^{T}=[-\overline{a}_0(-a_0\overline{a}_2+a),0,0,\cdots,0].$\\

\indent\hspace{-1cm} Now $det[\overline{B}_{n+1}^{\,\,T}-A_{n+1}B_{n+1}^{-1}\overline{A}_{n+1}^{\,\,T}]$=$det[E_1E_2\cdots E_nE_{n+1}]$\\
\indent\hspace{1cm}$= det[F_1E_2\cdots  F_{n-1}E_nF_{n+1}]$ + \,$det[G_1E_2\cdots  F_{n-1}E_nF_{n+1}]$ + \,$det[H_1E_2\cdots  F_{n-1}E_nF_{n+1}]$\\
\indent\hspace{1cm}+ \,$det[F_1E_2\cdots  F_{n-1}E_nG_{n+1}]$ + $det[G_1E_2\cdots  F_{n-1}E_nG_{n+1}]$ + \,$det[H_1E_2\cdots  F_{n-1}E_nG_{n+1}]$\\
\indent\hspace{1cm}+ \,$det[F_1E_2\cdots  G_{n-1}E_nF_{n+1}]$ + \,$det[G_1E_2\cdots G_{n-1}E_nF_{n+1}]$ + $det[H_1E_2\cdots G_{n-1}E_nF_{n+1}]$\\
\indent\hspace{1cm}+ \,$det[F_1E_2\cdots G_{n-1}E_nG_{n+1}]$ + \,$det[G_1E_2\cdots G_{n-1}E_nG_{n+1}]$ + \,$det[H_1E_2\cdots G_{n-1}E_nG_{n+1}].$\\
Out of these twelve determinants, only four given in Table 1 will be non-zero.
}\normalsize
Adding all these determinants we get\\
\vspace{.2cm}
\indent\hspace{2.3cm}$M_{n+1}=det[\overline{B}_{n+1}^{\,\,T}-A_{n+1}\,B_{n+1}^{-1}\,\overline{A}_{n+1}^{\,\,T}]\\\indent\hspace{3.3cm}=\displaystyle L_{n+1}\,8n>0,$ for $a \in \left({(n-2)}/{(n+2)},1\right)-\{n/(n+2)\}.$\\

{\bf Case 4.} \emph{When $k=n+1$ and $n$ is an even positive integer.} Here
\scriptsize{
\[\indent\hspace{-3cm}A_{n+1}=\begin{pmatrix}
a_{0}&a_{1}&a_{2} &\cdots &{a_n} \\
0&a_{0}&a_1&\cdots &a_{n-1} \\
0& 0&a_0&\cdots &a_{n-2}\\
\vdots & \vdots &\vdots& \ddots & \vdots\\
0&0&0&\cdots &a_{0}
\end{pmatrix}=\begin{pmatrix}
a_{0}&0&a_{2} &\cdots &a \\
0&a_{0}&0&\cdots &0 \\
0& 0&a_0&\cdots &0\\
\vdots & \vdots &\vdots& \ddots & \vdots\\
0&0&0&\cdots &a_{0}
\end{pmatrix}\quad and \quad B_{n+1}=\begin{pmatrix}
\overline{a}_{n+2}&\overline{a}_{n+1}&\overline {a}_{n} &\cdots &\overline{ a}_2 \\
0&\overline{a}_{n+2}&\overline{a}_{n+1}&\cdots &\overline{ a}_3 \\
0& 0&\overline{a}_{n+2}&\cdots &\overline{ a}_4\\
\vdots & \vdots &\vdots& \ddots & \vdots\\
0&0&0&\cdots &\overline{a}_{n+2}
\end{pmatrix}=\begin{pmatrix}
1&0&a &\cdots &\overline{a}_2 \\
0&1&0&\cdots &0 \\
0& 0&1&\cdots &0\\
\vdots & \vdots &\vdots& \ddots & \vdots\\
0&0&0&\cdots &1
\end{pmatrix}.\]}
\normalsize
 \indent\hspace{-2.2cm} One can verify that $\overline{B}_{n+1}^{\,\,T}-A_{n+1}B_{n+1}^{-1}\overline{A}_{n+1}^{\,\,T}=$
\scriptsize{
\[\indent\hspace{-3.2cm} \left( \begin{array}{llll}
1-a_0\overline{a}_0-\overline{a}_2(-aa_0+a_2)-a(-a_0\overline{a}_2+a)+(-a)^{\frac{n}{2}}(-aa_0+a_2)&\qquad \qquad0&-\overline{a}_0(-aa_0+a_2)+a\overline{a}_2(-aa_0+a_2)&\quad\cdots\\
\qquad0&1-a_0\overline{a}_0-\overline{a}_2(-aa_0+a_2)&\qquad\qquad 0 &\quad\cdots \\
a-a_0\overline{a}_2+(-a)^{\frac{n-2}{2}}(-aa_0+a_2)&\qquad\qquad 0 &1-a_0\overline{a}_0-\overline{a}_2(-aa_0+a_2)&\quad\cdots\\
\qquad\vdots &\qquad\qquad\vdots &\qquad\qquad\vdots &\quad \ddots \\
-a(-aa_0+a_2)&\qquad\qquad0&\qquad\qquad0&\quad\cdots\\
\qquad0&\qquad\qquad0&\qquad\qquad0&\quad\cdots\\
(-aa_0+a_2)&\qquad\qquad0&\qquad\qquad0&\quad\cdots
\end{array} \right.\]}

\[\indent\hspace{-1.1cm}\left.\begin{array}{clll}
a^{\frac{n-4}{2}}[-\overline{a}_0(-aa_0+a_2)+a\overline{a}_2(-aa_0+a_2)]-\overline{a}_2(-a_0\overline{a}_2+a)&\qquad 0 &(-a)^{\frac{n-2}{2}}[-\overline{a}_0(-aa_0+a_2)] -\overline{a}_0(-a_0\overline{a}_2+a) \\
\quad0 &(-a)^{\frac{n-4}{2}}[-\overline{a}_0(-aa_0+a_2)] &\qquad 0\\
a^{\frac{n-6}{2}}[-\overline{a}_0(-aa_0+a_2)+a\overline{a}_2(-aa_0+a_2)]&\qquad 0&(-a)^{\frac{n-4}{2}}[-\overline{a}_0(-aa_0+a_2)] \\
\vdots &\qquad\vdots &\qquad \vdots\\
1-a_0\overline{a}_0-\overline{a}_2(-aa_0+a_2)&\qquad0&[-\overline{a}_0(-aa_0+a_2)]\\
\qquad0 & 1-a_0\overline{a}_0&\qquad 0\\
a-a_0\overline{a}_2 &\qquad 0 &  1-a_0\overline{a}_0
\end{array} \right).\]\\
\normalsize  We split $E_1$, $E_{n-1}$ and $E_{n+1}$ in the following way (enteries in the remaining columns will be identical to those of the corresponding columns of Case 1).\\
$E_1=F_1+G_1+H_1$,  $E_{n-1}=F_{n-1}+G_{n-1}$, $E_{n+1}=F_{n+1}+G_{n+1},$ where\\
$F_1^{T}=(1-a_0\overline{a}_0-\overline{a}_2(-aa_0+a_2),0,a-a_0\overline{a}_2,0,\cdots,0),$\\
$G_1^{T}=((-a)^{\frac{n}{2}}(-aa_0+a_2),0,(-a)^{\frac{n-2}{2}}(-aa_0+a_2),0\cdots,(-aa_0+a_2)),$\\
\vspace{.3cm}
$H_1^{T}=(-a(-a_0\overline{a}_2+a),0,0,\cdots,0);$\\
$F_{n-1}^{T}=((-a)^{\frac{n-4}{2}}[-\overline{a}_0(-aa_0+a_2)+a\overline{a}_2(-aa_0+a_2)],0,(-a)^{\frac{n-6}{2}}[-\overline{a}_0(-aa_0+a_2)+\\\indent\hspace{1cm}a\overline{a}_2(-aa_0+a_2)],0\cdots, a-a_0\overline{a}_2),$\\
\vspace{.3cm}
$G_{n-1}^{T}=(-\overline{a}_2(-a_0\overline{a}_2+a),0,0,\cdots,0);$\\
$F_{n+1}^{T}=((-a)^{\frac{n-2}{2}}[-\overline{a}_0(-aa_0+a_2)],0,(-a)^{\frac{n-4}{2}}[-\overline{a}_0(-aa_0+a_2)],0,\cdots,1-a_0\overline{a}_0)$ and\\
$G_{n+1}^{T}=(-\overline{a}_0(-a_0\overline{a}_2+a),0,0,\cdots,0).$\\

As in Case 3,\,\, $M_{n+1}=det[E_1E_2\cdots E_nE_{n+1}]= det[F_1E_2\cdots  F_{n-1}E_nF_{n+1}] $\\
\indent\hspace{1cm}+\,$det[F_1E_2\cdots  F_{n-1}E_nG_{n+1}] +det[F_1E_2\cdots  G_{n-1}E_nF_{n+1}]+det[G_1E_2\cdots  F_{n-1}E_nF_{n+1}]$ \\
\indent\hspace{1cm}+ \,$det[G_1E_2\cdots  G_{n-1}E_nF_{n+1}]+det[G_1E_2\cdots  F_{n-1}E_nG_{n+1}]+det[H_1E_2\cdots  F_{n-1}E_nF_{n+1}],$\\
 as all the remaining determinants (total 5 in number) shall vanish. Values of non-zero determinants are given in Table 2.
\normalsize
Adding values of all these determinants we see that\\
when  $\displaystyle\frac{n}{2}$ is odd,
\vspace{.1cm}
\indent$M_{n+1}=\displaystyle L_{n+1}\, 8n(1-\cos\theta)$\\
\indent\hspace{3.8cm}$>0,$\,\,  for $\theta\not=2m\pi,\, m\in \mathbb{Z}$ and \\
when $\displaystyle \frac{n}{2}$ is even,
\vspace{.1cm}
\indent $M_{n+1}=\displaystyle L_{n+1}\,8n(1+\cos\theta)$\\
\indent\hspace{3.9cm}$>0,$ \,\, for $\theta\not=(2m+1)\pi,\, m\in \mathbb{Z}.$\\
Cases for $\theta=2m\pi\,\, {\rm and}\,\,\, \theta=(2m+1)\pi$ will be considered separately in the later part of this proof.
\begin{landscape}
{\bf Case 5.} \emph{When $k=n+2$ and $n$ is odd}. In this case, $\overline{B}_{n+2}^{\,\,T}-A_{n+2}B_{n+2}^{-1}\overline{A}_{n+2}^{\,\,T}=$
\scriptsize{
\[\indent\hspace{-5.5cm}\left( \begin{array}{llll}
1-a_0\overline{a}_0-\overline{a}_2(-aa_0+a_2)-a(-a_0\overline{a}_2+a)&(-a)^{\frac{n+1}{2}}(-aa_0+a_2)&-\overline{a}_0(-aa_0+a_2)+a\overline{a}_2(-aa_0+a_2)&\quad\cdots\\
(-a)^{\frac{n-1}{2}}(-aa_0+a_2)&1-a_0\overline{a}_0-\overline{a}_2(-aa_0+a_2)-a(-a_0\overline{a}_2+a)&\quad 0 &\quad\cdots \\
a-a_0\overline{a}_2&(-a)^{\frac{n-1}{2}}(-aa_0+a_2)&1-a_0\overline{a}_0-\overline{a}_2(-aa_0+a_2)&\quad\cdots\\
(-a)^{\frac{n-3}{2}}(-aa_0+a_2)&a-a_0\overline{a}_2&\qquad0&\quad\cdots\\
\qquad\vdots &\qquad\vdots &\qquad\vdots &\quad \ddots \\
\qquad0&-a(-aa_0+a_2)&\qquad0\\
(-aa_0+a_2)&\qquad0&\qquad0&\quad\cdots\\
\qquad0&(-aa_0+a_2)&\qquad0&\quad\cdots
\end{array} \right.\]}

\[\indent\hspace{1cm}\left.\begin{array}{clll}
-\overline{a}_2(-a_0\overline{a}_2+a)&a^{\frac{n-3}{2}}[-\overline{a}_0(-aa_0+a_2)+a\overline{a}_2(-aa_0+a_2)]& -\overline{a}_0(-a_0\overline{a}_2+a)&(-a)^{\frac{n-1}{2}}[-\overline{a}_0(-aa_0+a_2)] \\
(-a)^{\frac{n-5}{2}}[-\overline{a}_0(-aa_0+a_2)+a\overline{a}_2(-aa_0+a_2)]&-\overline{a}_2(-a_0\overline{a}_2+a)&(-a)^{\frac{n-3}{2}}[-\overline{a}_0(-aa_0+a_2)]&-\overline{a}_0(-a_0\overline{a}_2+a)\\
0&a^{\frac{n-5}{2}}[-\overline{a}_0(-aa_0+a_2)+a\overline{a}_2(-aa_0+a_2)]&\qquad 0&(-a)^{\frac{n-3}{2}}[-\overline{a}_0(-aa_0+a_2)]\\
(-a)^{\frac{n-7}{2}}[-\overline{a}_0(-aa_0+a_2)+a\overline{a}_2(-aa_0+a_2)]&\qquad0&(-a)^{\frac{n-5}{2}}[-\overline{a}_0(-aa_0+a_2)]&\qquad0\\
\vdots &\qquad\vdots &\qquad \vdots&\qquad \vdots\\
 0&1-a_0\overline{a}_0-\overline{a}_2(-aa_0+a_2)&\qquad 0&[-\overline{a}_0(-aa_0+a_2)]\\
a-a_0\overline{a}_2 &\qquad 0 &  1-a_0\overline{a}_0&\qquad0\\
0&a-a_0\overline{a}_2&\qquad0&1-a_0\overline{a}_0
\end{array} \right)\]\\
\normalsize
\noindent If $E_j$ $(j=1,2,\cdots,n+2)$ are the columns of $\overline{B}_{n+2}^{\,\,T}-A_{n+2}B_{n+2}^{-1}\overline{A}_{n+2}^{\,\,T}$ then we split $E_1$,\,$E_2$,\,$E_{n-1}$,\,$E_n$,\,$E_{n+1}$ and $E_{n+2}$ as under (enteries in the remaining columns will be identical to those of the corresponding columns of Case 1).\\
$E_1=F_1+G_1+H_1$, $E_2=F_2+G_2+H_2$, $E_{n-1}=F_{n-1}+G_{n-1}$, $E_n=F_n+G_n$,  $E_{n+1}=F_{n+1}+G_{n+1}$ and $E_{n+2}=F_{n+2}+G_{n+2},$ where\\
$F_1^{T}=(1-a_0\overline{a}_0-\overline{a}_2(-aa_0+a_2),0,a-a_0\overline{a}_2,0,\cdots,0),$\\
$G_1^{T}=(0,(-a)^{\frac{n-1}{2}}(-aa_0+a_2),0,(-a)^{\frac{n-3}{2}}(-aa_0+a_2),0\cdots,(-aa_0+a_2),0),$\\
\vspace{.3cm}
$H_1^{T}=(-a(-a_0\overline{a}_2+a),0,0,\cdots,0);$\\
\end{landscape}
\noindent$F_2^{T}=(0,1-a_0\overline{a}_0-\overline{a}_2(-aa_0+a_2),0,a-a_0\overline{a}_2,0,\cdots,0),$\\
$G_2^{T}=((-a)^{\frac{n+1}{2}}(-aa_0+a_2),0,(-a)^{\frac{n-1}{2}}(-aa_0+a_2),0,(-a)^{\frac{n-3}{2}}(-aa_0+a_2),0\cdots,(-aa_0+a_2)),$\\
\vspace{.3cm}
 $H_2^{T}=(0,-a(-a_0\overline{a}_2+a),0,0,\cdots,0);$\\
$F_{n-1}^{T}=(0,(-a)^{\frac{n-5}{2}}[-\overline{a}_0(-aa_0+a_2)+a\overline{a}_2(-aa_0+a_2)],0,(-a)^{\frac{n-7}{2}}[-\overline{a}_0(-aa_0+a_2)+\\\indent\hspace{1cm}a\overline{a}_2(-aa_0+a_2)],0\cdots, a-a_0\overline{a}_2,0),$\\
\vspace{.3cm}
$G_{n-1}^{T}=(-\overline{a}_2(-a_0\overline{a}_2+a),0,0,\cdots,0);$\\
$F_{n}^{T}=((-a)^{\frac{n-3}{2}}[-\overline{a}_0(-aa_0+a_2)+a\overline{a}_2(-aa_0+a_2)],0,(-a)^{\frac{n-5}{2}}[-\overline{a}_0(-aa_0+a_2)+\\\indent\hspace{1cm}a\overline{a}_2(-aa_0+a_2)],0\cdots, a-a_0\overline{a}_2),$\\
\vspace{.3cm}
$G_{n}^{T}=(0,-\overline{a}_2(-a_0\overline{a}_2+a),0,0,\cdots,0);$\\
$F_{n+1}^{T}=(0,(-a)^{\frac{n-3}{2}}[-\overline{a}_0(-aa_0+a_2)],0,(-a)^{\frac{n-5}{2}}[-\overline{a}_0(-aa_0+a_2)],0,\cdots,1-a_0\overline{a}_0,0),$\\
\vspace{.3cm}
$G_{n+1}^{T}=(-\overline{a}_0(-a_0\overline{a}_2+a),0,0,\cdots,0)$ and \\
$F_{n+2}^{T}=((-a)^{\frac{n-1}{2}}[-\overline{a}_0(-aa_0+a_2)],0,(-a)^{\frac{n-3}{2}}[-\overline{a}_0(-aa_0+a_2)],0,(-a)^{\frac{n-5}{2}}[-\overline{a}_0(-aa_0+a_2)\\\indent\hspace{1cm}0,\cdots,1-a_0\overline{a}_0),$\\
$G_{n+2}^{T}=(0,-\overline{a}_0(-a_0\overline{a}_2+a),0,0,\cdots,0).$\\
As a consequence of above splitting of columns, $M_{n+2}$ can be written as a sum of $144$ determinants out of which $123$ shall vanish and values of remaining $21$ non-zero determinants are listed in Table 3. Adding all these determinants we get\\
\vspace{.3cm}
\indent\hspace{2.3cm}$M_{n+2}=det[\overline{B}_{n+2}^{\,\,T}-A_{n+2}\,B_{n+2}^{-1}\,\overline{A}_{n+2}^{\,\,T}]\\
\vspace{.3cm}
\indent\hspace{3.3cm}=\displaystyle L_{n+2}\,[8(1+\cos2\theta)]$\\
\vspace{.3cm}
\indent\hspace{3.4cm}$>0,$  for $\theta\not=(2m+1)\frac{\pi}{2},\, m\in \mathbb{Z} .$\\
The case when $\theta=(2m+1)\frac{\pi}{2}$ will be considered separately.\\
\begin{landscape}
{\bf Case 6.} \emph{When $k=n+2$ and $n$ is even}. In this case, $\overline{B}_{n+2}^{\,\,T}-A_{n+2}B_{n+2}^{-1}\overline{A}_{n+2}^{\,\,T}$=
\scriptsize{
\[\indent\hspace{-5.3cm}\left( \begin{array}{llll}
1-a_0\overline{a}_0-\overline{a}_2(-aa_0+a_2)-a(-a_0\overline{a}_2+a)+(-a)^{\frac{n}{2}}(-aa_0+a_2)&\qquad0&-\overline{a}_0(-aa_0+a_2)+a\overline{a}_2(-aa_0+a_2)&\cdots\\
\qquad0&1-a_0\overline{a}_0-\overline{a}_2(-aa_0+a_2)-a(-a_0\overline{a}_2+a)+(-a)^{\frac{n}{2}}(-aa_0+a_2)&0&\cdots \\
a-a_0\overline{a}_2+(-a)^{\frac{n-2}{2}}(-aa_0+a_2)&\qquad0&1-a_0\overline{a}_0-\overline{a}_2(-aa_0+a_2)&\cdots\\
\qquad0&a-a_0\overline{a}_2+(-a)^{\frac{n-2}{2}}(-aa_0+a_2)&0&\cdots\\
\qquad\vdots &\qquad\vdots&\vdots &\ddots  \\
\qquad0& -a(-aa_0+a_2)&0&\cdots\\
(-aa_0+a_2)&\qquad0&0 &\cdots\\
\qquad0&(-aa_0+a_2)&0&\cdots
\end{array} \right.\]

\[\indent\hspace{-2.5cm}\left.\begin{array}{ccll}
(-a)^{\frac{n-4}{2}}[-\overline{a}_0(-aa_0+a_2)+a\overline{a}_2(-aa_0+a_2)]-\overline{a}_2(-a_0\overline{a}_2+a)&0&(-a)^{\frac{n-2}{2}}[-\overline{a}_0(-aa_0+a_2)]-\overline{a}_0(-a_0\overline{a}_2+a)&\qquad0 \\
\qquad0&(-a)^{\frac{n-4}{2}}[-\overline{a}_0(-aa_0+a_2)+a\overline{a}_2(-aa_0+a_2)]-\overline{a}_2(-a_0\overline{a}_2+a)&0&\indent\hspace{-2.3cm}(-a)^{\frac{n-2}{2}}[-\overline{a}_0(-aa_0+a_2)]-\overline{a}_0(-a_0\overline{a}_2+a)\\
(-a)^{\frac{n-6}{2}}[-\overline{a}_0(-aa_0+a_2)+a\overline{a}_2(-aa_0+a_2)]&0&(-a)^{\frac{n-4}{2}}[-\overline{a}_0(-aa_0+a_2)]&\qquad0\\
\qquad0&(-a)^{\frac{n-6}{2}}[-\overline{a}_0(-aa_0+a_2)+a\overline{a}_2(-aa_0+a_2)]&0&(-a)^{\frac{n-4}{2}}[-\overline{a}_0(-aa_0+a_2)]\\
\qquad \vdots&\vdots &\vdots&\qquad \vdots\\
\qquad0&1-a_0\overline{a}_0-\overline{a}_2(-aa_0+a_2)& 0&[-\overline{a}_0(-aa_0+a_2)]\\
 \quad a-a_0\overline{a}_2&0 &  1-a_0\overline{a}_0&\qquad0\\
\qquad0&a-a_0\overline{a}_2&0&1-a_0\overline{a}_0
\end{array} \right).\]}\\

\normalsize
\indent\hspace{-.6cm} As in the last case, here we split the columns $E_1$,\,$E_2$,\,$E_{n-1}$,\,$E_n$,\,$E_{n+1}$ and $E_{n+2}$ as
$E_1=F_1+G_1+H_1$, $E_2=F_2+G_2+H_2$, $E_{n-1}=F_{n-1}+G_{n-1}$, $E_n=F_n+G_n$,  $E_{n+1}=F_{n+1}+G_{n+1}$ and $E_{n+2}=F_{n+2}+G_{n+2},$ where\\
$F_1^{T}=(1-a_0\overline{a}_0-\overline{a}_2(-aa_0+a_2),0,a-a_0\overline{a}_2,0,\cdots,0),$\\
$G_1^{T}=((-a)^{\frac{n}{2}}(-aa_0+a_2),0,(-a)^{\frac{n-2}{2}}(-aa_0+a_2),0\cdots,(-aa_0+a_2),0),$\\
$H_1^{T}=(-a(-a_0\overline{a}_2+a),0,0,\cdots,0);$\\
\end{landscape}
\noindent$F_2^{T}=(0,1-a_0\overline{a}_0-\overline{a}_2(-aa_0+a_2),0,a-a_0\overline{a}_2,0,\cdots,0),$\\
$G_2^{T}=(0,(-a)^{\frac{n}{2}}(-aa_0+a_2),0,(-a)^{\frac{n-1}{2}}(-aa_0+a_2),0,(-a)^{\frac{n-2}{2}}(-aa_0+a_2),0\cdots,(-aa_0+a_2)),$\\
\vspace{.3cm}
$H_2^{T}=(0,-a(-a_0\overline{a}_2+a),0,0,\cdots,0);$\\
\noindent$F_{n-1}^{T}=((-a)^{\frac{n-4}{2}}[-\overline{a}_0(-aa_0+a_2)+a\overline{a}_2(-aa_0+a_2)],0,(-a)^{\frac{n-6}{2}}[-\overline{a}_0(-aa_0+a_2)+\\\indent\hspace{1cm}a\overline{a}_2(-aa_0+a_2)],0\cdots, a-a_0\overline{a}_2,0),$\\
\vspace{.3cm}
$G_{n-1}^{T}=(-\overline{a}_2(-a_0\overline{a}_2+a),0,0,\cdots);$\\
$F_{n}^{T}=(0,(-a)^{\frac{n-4}{2}}[-\overline{a}_0(-aa_0+a_2)+a\overline{a}_2(-aa_0+a_2)],0,(-a)^{\frac{n-6}{2}}[-\overline{a}_0(-aa_0+a_2)+\\\indent\hspace{1cm}a\overline{a}_2(-aa_0+a_2)],0\cdots, a-a_0\overline{a}_2),$\\
\vspace{.3cm}
$G_{n}^{T}=(0,-\overline{a}_2(-a_0\overline{a}_2+a),0,0,\cdots);$\\
$F_{n+1}^{T}=((-a)^{\frac{n-2}{2}}[-\overline{a}_0(-aa_0+a_2)],0,(-a)^{\frac{n-4}{2}}[-\overline{a}_0(-aa_0+a_2)],0,\cdots,1-a_0\overline{a}_0,0),$\\
\vspace{.3cm}
$G_{n+1}^{T}=(-\overline{a}_0(-a_0\overline{a}_2+a),0,0,\cdots,0)$ and\\
$F_{n+2}^{T}=(0,(-a)^{\frac{n-2}{2}}[-\overline{a}_0(-aa_0+a_2)],0,(-a)^{\frac{n-4}{2}}[-\overline{a}_0(-aa_0+a_2)],0,\\\indent\hspace{1cm}(-a)^{\frac{n-6}{2}}[-\overline{a}_0(-aa_0+a_2)],0,\cdots,1-a_0\overline{a}_0),$\\
$G_{n+2}^{T}=(0,-\overline{a}_0(-a_0\overline{a}_2+a),0,0,\cdots,0].$\\
In this case, $M_{n+2}$ is a sum of $144$ determinants out of which only $49$ listed in Table 4 are non-zero. Adding values of all these determinants we obtain that,\\
\vspace{.1cm}
 when $\displaystyle\frac{n}{2}$ is odd,
\indent $M_{n+2}=det[\overline{B}_{n+2}^{\,\,T}-A_{n+2}\,B_{n+2}^{-1}\,\overline{A}_{n+2}^{\,\,T}]\\
\vspace{.1cm}
\indent\hspace{3.9cm}=\displaystyle L_{n+2}\,[16(1-\cos\theta)^2]$\\
\vspace{.1cm}
\indent\hspace{4cm}$>0,$  for $\theta\not=2m\pi,\, m\in \mathbb{Z}$ and \\
\vspace{.1cm}
when $\displaystyle \frac{n}{2}$ is even,
\indent $M_{n+2}=\displaystyle L_{n+2}\,[16(1+\cos\theta)^2]$\\
\indent\hspace{4.1cm}$>0,$  for $\theta\not=(2m+1)\pi,\, m\in \mathbb{Z} .$\\
Now we discuss the cases when $\theta=(2m+1)\pi,\,\,2m\pi\,\,{\rm and}\,\, (2m+1)\pi/2, m\in \mathbb{Z}.$\\
When $\theta=(2m+1)\pi, m\in \mathbb{Z},$ and both $n, n/2$ are even, then the polynomial $p$ given in (6) reduces to:\\
\indent\hspace{2cm}$ p(z)= z^{n+2}+az^n- \frac{1}{2}\left(2+an-n\right)z^2-\frac{1}{2}\left(2a+an-n\right).$\\
\indent\hspace{2.8cm}$=(z^2+1)[z^n+(a-1)z^{n-2}-(a-1)z^{n-4}+\cdots+(a-1)z^2-\frac{1}{2}(2a+an-n)]$\\
$\indent\hspace{2.7cm}=(z^2+1)q(z).$\\
It suffices to show that zeros of $q$ lie inside or on $|z|=1$. Since $|\frac{1}{2}(2a+an-n)|<1$ whenever  $a \in \left({(n-2)}/{(n+2)},1\right)-\{n/(n+2)\}$, by applying Lemma \vspace{.1cm}
2.1 on $q,$ after comparing it with $t$, we get \\
$\indent\displaystyle q_1(z)=\frac{\overline{a}_nq(z)-a_0q^*(z)}{z}$\\
$\indent\hspace{.8cm}=(1-a)(1+\frac{1}{2}(2a+an-n))z\left\{(1+\frac{n}{2})z^{n-2}-z^{n-4}+z^{n-6}-\cdots+z^2-1\right\}.$\\
$\indent\hspace{.8cm}=(1-a)(1+\frac{1}{2}(2a+an-n))z r_1(z)$,\\ where \\
$\indent\hspace{0cm}r_1(z)=(1+\frac{n}{2})z^{n-2}-z^{n-4}+z^{n-6}-\cdots+z^2-1.$\\
The number of zeros of $q_1$ inside the unit circle is one less than the number of zeros of $q$ inside the unit circle and so the number of zeros of $r_1$ inside the unit circle is two less than the number of zeros of $q$ inside the unit circle. As $1<1+\frac{n}{2},$ therefore by applying Lemma 2.1 again on $r_1$ we get\\
$\indent\displaystyle q_2(z)= z \frac{n}{2}[(2+\frac{n}{2})z^{n-4}-z^{n-6}+z^{n-8}-\cdots-z^2+1].$\\
$\indent\displaystyle\hspace{.8cm} = z r_2(z)$.\\
 The number of zeros of $r_2$ inside the circle $|z|=1$ is four less than the number of zeros of $q$ inside the unit circle. Similarly using Lemma 2.1 on $r_2$ we have,\\
$\indent\hspace{1cm}q_3(z)= z(1+\frac{n}{2})[(3+\frac{n}{2})z^{n-6}-z^{n-8}+z^{n-10}-\cdots+z^2-1].$\\
Continuing in this manner we derive that \\
$\indent\hspace{-.1cm}\displaystyle q_k(z)= z\left[(k-2)+\frac{n}{2}\right]\left\{(k+\frac{n}{2})z^{n-2k}-z^{n-2(k+1)}+\cdots+(-1)^{k+1} z^2+(-1)^{k}\right\},\\$
$\indent\hspace{-.8cm}\indent\displaystyle\hspace{.9cm} = z\, r_k(z),\,\, k=2,3,\cdots,(n/2-1)$.\\
The number of zeros of $r_k$ inside the unit circle is $2k$ less than the number of zeros of $q$ inside the unit circle and in particular, for $k=(n/2-1),$\\
$\indent\displaystyle q_{n/2-1}(z)=(n-3)\{(n-1)z^2-1\},$\\ which has $(n-2)$ less number of zeros inside the unit circle than the number of zeros of $q$ inside the unit circle. But the zeros of $q_{n/2-1}$ are $\displaystyle\pm\frac{1}{\sqrt{n-1}}$ which lie inside the unit circle $|z|=1$. Consequently all the zeros of $q$ lie inside $|z|=1$.\\
\indent When $\theta=2m\pi$, $n$ is even and ${n}/{2}$ is odd, then the polynomial $p(z)=(z^2+1)\eta(z),$ where\\
$\eta(z)=(z^2+1)[z^n+(a-1)z^{n-2}-(a-1)z^{n-4}+\cdots-(a-1)z^2+\frac{1}{2}(2a+an-n)].$\\
Proceeding on the same lines as above we see that all the zeros of $p$ lie inside or on the circle $|z|=1$.\\
\indent When $\displaystyle \theta=(2m+1){\pi}/{2}$ and $n$ is odd, then\\
(i) If $n=4s+1,\,\, s=1,2,3,\cdots$ in this case $p(z)=(z+i)\zeta(z)$ and \\
$\zeta(z)=[z^{n+1}-iz^{n}+(a-1)z^{n-1}-i(a-1)z^{n-2}-(a-1)z^{n-3}+\cdots+\frac{i}{2}(2a+an-n)+\frac{1}{2}(2a+an-n)].$\\
(ii) If $n=4s-1,\,\, s=2,3,4,\cdots $ in this case
$p(z)=(z-i)\xi(z)$ here \\
$\xi(z)=[z^{n+1}+iz^{n}+(a-1)z^{n-1}+i(a-1)z^{n-2}-(a-1)z^{n-3}-i(a-1)z^{n-4}+\cdots+\\\indent\hspace{1cm} \frac{i}{2}(2a+an-n)-\frac{1}{2}(2a+an-n)].$\\
Again proceeding as above we conclude that all the zeros of $p$ lie inside or on the circle $|z|=1$.\\
\indent Lastly, we take up the case when $a=n/(n+2)$. In this case
$$\displaystyle\widetilde{\omega}(z)=\displaystyle z^{n+2}e^{2i\theta}\left[\frac{(n+2)z^{n}+n z^{n-2}+ 2e^{-i\theta}}{2e^{i\theta}z^n+nz^2+(n+2)}\right]=\displaystyle z^{n+2}e^{2i\theta}\frac{\beta(z)}{\beta^*(z)},$$
\noindent where,\,\, $\beta(z)=(n+2)z^{n}+n z^{n-2}+ 2e^{-i\theta}$. To prove that $|\widetilde{\omega}(z)|<1$ in $E$ it suffices to show that all the zeros of $\beta$ lie inside or on $|z|=1$. The repeated application of Lemma 2.1 (as in the exceptional cases discussed above) allows us to conclude that this is infact true. We skip the details for want of space.\\
\indent This completes the proof.
\end{proof}
\bigskip

\begin{table}[H]
\caption{ Values of non-zero determinants in Case 3.}
\bigskip
\indent\hspace{2cm}\begin{tabular}{|c|l|l|}
\hline
\textbf{\emph{Determinant}}& \hspace{1.5cm}\textbf{\emph{Value}}\\
\hline
 $det[F_1E_2\cdots F_{n-1}E_nF_{n+1}] $&$L_{n+1}[(2n+1)^2] $\\
    \hline
$det[H_1E_2\cdots F_{n-1}E_nF_{n+1}]$  &$L_{n+1}[-a(2n+1)(2n-1)]$ \\
    \hline
$det[G_1E_2\cdots F_{n-1}E_nG_{n+1}] $ &$L_{n+1}[\frac{1}{2}(2n-1)(n+1)(2a+an-n)]$ \\
    \hline
$ det[G_1E_2\cdots G_{n-1}E_nF_{n+1}] $  & $L_{n+1}[\frac{1}{2}(2n-1)(n-1)(n-2-an)] $ \\
    \hline
\end{tabular}
\end{table}
\begin{table}[H]
\caption{Values of non-zero determinants in Case 4.}
\bigskip
\begin{tabular}{|c|l|l|}
\hline
 \textbf{\emph{Determinant}}& \multicolumn{2}{|c|}{\textbf{\emph{Value}}}\\ \cline{2-3}
  &\hspace{1cm}$\displaystyle when\,\,\bf{\frac{n}{2}}$\,\, $is\,\, odd$ &\hspace{1cm} $\displaystyle when\,\,\bf{\frac{n}{2}} $\,\, $is\,\, even$\\
  \hline
 $ det[F_1E_2\cdots F_{n-1}E_nF_{n+1}] $& $L_{n+1}[2n(2n+2)] $&$L_{n+1}[2n(2n+2)] $\\
    \hline
 $ det[F_1E_2\cdots F_{n-1}E_nG_{n+1}] $  & $L_{n+1}[e^{i\theta}(2a+an-n)n^2] $&$L_{n+1}[-e^{i\theta}(2a+an-n)n^2]$ \\
    \hline
$ det[F_1E_2\cdots G_{n-1}E_nF_{n+1}] $  & $L_{n+1}[e^{i\theta}(n-an-2)n(n+2)] $&$L_{n+1}[-e^{i\theta}(n-an-2)n(n+2)] $ \\
    \hline
$ det[G_1E_2\cdots F_{n-1}E_nF_{n+1}] $  & $L_{n+1}[-e^{-i\theta} 4n] $&$L_{n+1}[e^{-i\theta} 4n]$ \\
    \hline
$ det[G_1E_2\cdots G_{n-1}E_nF_{n+1}] $& $L_{n+1}[n(n-2)(n-an-2)] $&$L_{n+1}[n(n-2)(n-an-2)] $\\
    \hline
$ det[G_1E_2\cdots F_{n-1}E_nG_{n+1}] $& $L_{n+1}[(2a+an-n)n^2] $&$L_{n+1}[(2a+an-n)n^2] $\\
    \hline
$ det[H_1E_2\cdots F_{n-1}E_nF_{n+1}] $  & $L_{n+1}[-a (2n)^2] $&$L_{n+1}[-a (2n)^2]$ \\
    \hline
\end{tabular}
\end{table}
\begin{table}[H]
\caption {Values of non-zero determinants in Case 5.}
\bigskip
\indent\hspace{1cm}\begin{tabular}{|c|l|l|}
  \hline
\textbf{\emph{Determinant}}& \hspace{1.5cm}\textbf{\emph{Value}}\\
  \hline
 $ det[F_1F_2\cdots F_{n-1}F_nF_{n+1}F_{n+2}] $& $L_{n+2}\,[(2n+1)(2n+3)] $\\
    \hline
 $det[F_1F_2\cdots F_{n-1}F_nG_{n+1}G_{n+2}] $  & $L_{n+2}\,[\frac{e^{2i\theta}}{4}n^2(n-2a-an)^2] $ \\
    \hline
$ det[F_1F_2\cdots F_{n-1}G_nG_{n+1}F_{n+2}] $  & $L_{n+2}\,[-\frac{e^{2i\theta}}{4}n(n+2)(n-2a-an)(n-2-an)] $ \\
    \hline
$  det[F_1F_2\cdots G_{n-1}F_nF_{n+1}G_{n+2}] $  & $L_{n+2}\,[-\frac{e^{2i\theta}}{4}n(n+2)(n-2a-an)(n-2-an)]  $ \\
    \hline
$ det[F_1F_2\cdots G_{n-1}G_nF_{n+1}F_{n+2}] $& $L_{n+2}\,[\frac{e^{2i\theta}}{4}(n+2)^2(n-2-an)^2]  $\\
    \hline
$ det[F_1G_2\cdots F_{n-1}F_nF_{n+1}G_{n+2}] $  & $L_{n+2}\,[-\frac{1}{2}(n-2a-an)(2n-1)(n+3)] $ \\
    \hline
$ det[F_1G_2\cdots F_{n-1}G_nF_{n+1}F_{n+2}] $  & $L_{n+2}\,[\frac{1}{2}(n-2-an)(2n-1)(n+1)] $ \\
\hline
$ det[F_1H_2\cdots F_{n-1}F_nF_{n+1}F_{n+2}] $  & $L_{n+2}\,[-a(2n-1)(2n+3)] $ \\
    \hline
\end{tabular}
\end{table}
\begin{table}[H]
\begin{tabular}{|c|l|l|}
\hline
$ det[G_1F_2\cdots F_{n-1}F_nG_{n+1}F_{n+2}] $  & $L_{n+2}\,[-\frac{1}{2}(n-2a-an)(2n+1)(n+1)] $ \\
    \hline
$ det[G_1F_2\cdots G_{n-1}F_nF_{n+1}F_{n+2}] $  & $L_{n+2}\,[\frac{1}{2}(n-2-an)(2n+1)(n-1)] $ \\
    \hline
$ det[G_1G_2\cdots F_{n-1}F_nF_{n+1}F_{n+2}] $  & $L_{n+2}\,[4e^{-2i\theta}] $ \\
    \hline
 $ det[G_1G_2\cdots F_{n-1}F_nG_{n+1}G_{n+2}] $  & $L_{n+2}\,[\frac{1}{4}(n-2a-an)^2(n^2-1)] $ \\
    \hline
 $ det[G_1G_2\cdots F_{n-1}G_nG_{n+1}F_{n+2}] $  & $L_{n+2}\,[-\frac{1}{4}(n-2a-an)(n-2-an)(n-1)^2] $ \\
    \hline
  $ det[G_1G_2\cdots G_{n-1}F_nF_{n+1}G_{n+2}] $  & $L_{n+2}\,[-\frac{1}{4}(n-2a-an)(n-2-an)(n-3)(n+1)] $ \\
    \hline
  $ det[G_1G_2\cdots G_{n-1}G_nF_{n+1}F_{n+2}] $  & $L_{n+2}\,[\frac{1}{4}(n-2-an)^2(n-3)(n-1)] $ \\
    \hline
  $ det[G_1H_2\cdots F_{n-1}F_nG_{n+1}F_{n+2}] $  & $L_{n+2}\,[\frac{1}{2}a(n-2a-an)(n-1)(2n+1)] $ \\
    \hline
  $ det[G_1H_2\cdots G_{n-1}F_nF_{n+1}F_{n+2}] $  & $L_{n+2}\,[-\frac{1}{2}a(n-2-an)(n-3)(2n+1)] $ \\
    \hline
  $ det[H_1F_2\cdots F_{n-1}F_nF_{n+1}F_{n+2}] $  & $L_{n+2}\,[-a(2n+1)^2] $ \\
    \hline
 $ det[H_1G_2\cdots F_{n-1}F_nF_{n+1}G_{n+2}] $  & $L_{n+2}\,[\frac{1}{2}a(n-2a-an)(n+1)(2n-1)] $ \\
    \hline
 $ det[H_1G_2\cdots F_{n-1}G_nF_{n+1}F_{n+2}] $  & $L_{n+2}\,[-\frac{1}{2}a(n-2-an)(n-1)(2n-1)] $ \\
    \hline
$ det[H_1H_2\cdots F_{n-1}F_nF_{n+1}F_{n+2}] $  & $L_{n+2}\,[a^2(2n+1)(2n-1)] $ \\
   \hline
\end{tabular}
\end{table}
\begin{table}[H]
\caption {Values of non-zero determinants in Case 6.}
\bigskip
\indent\hspace{-2.6cm}\begin{tabular}{|p{5cm}|p{7cm}|p{7cm}|}
\hline
 \textbf{\emph{Determinant}}& \multicolumn{2}{|c|}{\textbf{\emph{Value}}}\\ \cline{2-3}
  &\hspace{1cm}$\displaystyle when\,\,\bf{\frac{n}{2}}$\,\, $is\,\,odd$ &\hspace{1cm} $\displaystyle when\,\,\bf{\frac{n}{2}} $\,\, $is\,\, even$\\
  \hline
 $ det[F_1F_2\cdots F_{n-1}F_nF_{n+1}F_{n+2}] $& $L_{n+2}\,[(2n+2)^2] $&$L_{n+2}\,[(2n+2)^2] $\\
    \hline
 $ det[F_1F_2\cdots F_{n-1}F_nF_{n+1}G_{n+2}] $& $L_{n+2}\,[-\frac{e^{i\theta}}{2}(n-2a-an)n(2n+2)] $&$L_{n+2}\,[\frac{e^{i\theta}}{2}(n-2a-an)n(2n+2)]$\\
    \hline
 $ det[F_1F_2\cdots F_{n-1}F_nG_{n+1}F_{n+2}] $& $L_{n+2}\,[-\frac{e^{i\theta}}{2}(n-2a-an)n(2n+2)] $&$ L_{n+2}\,[\frac{e^{i\theta}}{2}(n-2a-an)n(2n+2)]$\\
    \hline
$ det[F_1F_2\cdots F_{n-1}F_nG_{n+1}G_{n+2}] $& $L_{n+2}\,[\frac{e^{2i\theta}}{4}(n-2a-an)^2n^2] $&$ L_{n+2}\,[\frac{e^{2i\theta}}{4}(n-2a-an)^2n^2]$\\
    \hline
$ det[F_1F_2\cdots F_{n-1}G_nF_{n+1}F_{n+2}] $& $L_{n+2}\,[\frac{e^{i\theta}}{2}(n-2-an)(2n+2)(n+2)] $&$L_{n+2}\,[-\frac{e^{i\theta}}{2}(n-2-an)(2n+2)(n+2)] $\\
    \hline
$ det[F_1F_2\cdots F_{n-1}G_nG_{n+1}F_{n+2}] $& $L_{n+2}\,[-\frac{e^{2i\theta}}{4}(n-2a-an)(n-2-an)(n+2)n] $& $L_{n+2}\,[-\frac{e^{2i\theta}}{4}(n-2a-an)(n-2-an)(n+2)n] $\\
    \hline
$ det[F_1F_2\cdots G_{n-1}F_nF_{n+1}F_{n+2}] $& $L_{n+2}\,[\frac{e^{i\theta}}{2}(n-2-an)(2n+2)(n+2)] $&$L_{n+2}\,[-\frac{e^{i\theta}}{2}(n-2-an)(2n+2)(n+2)] $\\
    \hline
$ det[F_1F_2\cdots G_{n-1}F_nF_{n+1}G_{n+2}] $& $L_{n+2}\,[-\frac{e^{2i\theta}}{4}(n-2a-an)(n-2-an)(n+2)n] $& $L_{n+2}\,[-\frac{e^{2i\theta}}{4}(n-2a-an)(n-2-an)(n+2)n]$\\
    \hline
$ det[F_1F_2\cdots G_{n-1}G_nF_{n+1}F_{n+2}] $& $L_{n+2}\,[\frac{e^{2i\theta}}{4}(n-2-an)^2(n+2)^2] $&$ L_{n+2}\,[\frac{e^{2i\theta}}{4}(n-2-an)^2(n+2)^2]$\\
    \hline
$ det[F_1G_2\cdots F_{n-1}F_nF_{n+1}F_{n+2}] $& $L_{n+2}\,[-2e^{-i\theta}(2n+2)] $&$ L_{n+2}\,[2e^{-i\theta}(2n+2)]$\\
    \hline
$ det[F_1G_2\cdots F_{n-1}F_nF_{n+1}G_{n+2}] $& $L_{n+2}\,[-(n-2a-an)n(n+1)] $&$ L_{n+2}\,[-(n-2a-an)n(n+1)]$\\
    \hline
$ det[F_1G_2\cdots F_{n-1}F_nG_{n+1}F_{n+2}] $& $L_{n+2}\,[(n-2a-an)n] $&$ L_{n+2}\,[(n-2a-an)n]$\\
    \hline
$ det[F_1G_2\cdots F_{n-1}F_nG_{n+1}G_{n+2}] $& $L_{n+2}\,[\frac{e^{i\theta}}{4}(n-2a-an)^2n^2] $&$ L_{n+2}\,[-\frac{e^{i\theta}}{4}(n-2a-an)^2n^2]$\\
 \hline
$ det[F_1G_2\cdots F_{n-1}G_nF_{n+1}F_{n+2}] $& $L_{n+2}\,[(n-2-an)(n-2)(n+1)] $&$ L_{n+2}\,[(n-2-an)(n-2)(n+1)]$\\
    \hline
$ det[F_1G_2\cdots F_{n-1}G_nG_{n+1}F_{n+2}] $& $L_{n+2}\,[-\frac{e^{i\theta}}{4}(n-2a-an)(n-2-an)(n-2)n] $& $L_{n+2}\,[\frac{e^{2i\theta}}{4}(n-2a-an)(n-2-an)(n-2)n] $\\
    \hline
$ det[F_1G_2\cdots G_{n-1}F_nF_{n+1}F_{n+2}] $& $L_{n+2}\,[-(n-2-an)(n+2)] $&$ L_{n+2}\,[-(n-2-an)(n+2)]$\\
    \hline
$ det[F_1G_2\cdots G_{n-1}F_nF_{n+1}G_{n+2}] $& $L_{n+2}\,[-\frac{e^{i\theta}}{4}(n-2a-an)(n-2-an)(n+2)n] $&$L_{n+2}\,[\frac{e^{i\theta}}{4}(n-2a-an)(n-2-an)(n+2)n] $\\
\hline
\end{tabular}
\end{table}
\indent\hspace{-3.2cm}\begin{tabular}{|p{5cm}|p{7cm}|p{7cm}|}
\hline
$ det[F_1G_2\cdots G_{n-1}G_nF_{n+1}F_{n+2}] $& $L_{n+2}\,[\frac{e^{i\theta}}{4}(n-2-an)^2(n-2)(n+2)] $&$L_{n+2}\,[-\frac{e^{i\theta}}{4}(n-2-an)^2(n-2)(n+2)] $\\
 \hline
$ det[F_1H_2\cdots F_{n-1}F_nF_{n+1}F_{n+2}] $& $L_{n+2}\,[-4an(n+1)] $&$L_{n+2}\,[-4an(n+1)] $\\
    \hline
$ det[F_1H_2\cdots F_{n-1}F_nG_{n+1}F_{n+2}] $& $L_{n+2}\,[e^{i\theta}a(n-2a-an)n^2] $&$L_{n+2}\,[-e^{i\theta}a(n-2a-an)n^2] $\\
    \hline
$ det[F_1H_2\cdots G_{n-1}F_nF_{n+1}F_{n+2}] $& $L_{n+2}\,[-e^{i\theta}a(n-2-an)n(n+2)] $&$L_{n+2}\,[e^{i\theta}a(n-2-an)n(n+2)] $\\
    \hline
$ det[G_1F_2\cdots F_{n-1}F_nF_{n+1}F_{n+2}] $& $L_{n+2}\,[-4e^{-i\theta}(n+1)] $&$L_{n+2}\,[4e^{-i\theta}(n+1)] $\\
    \hline
$ det[G_1F_2\cdots F_{n-1}F_nF_{n+1}G_{n+2}] $& $L_{n+2}\,[(n-2a-an)n] $&$L_{n+2}\,[(n-2a-an)n] $\\
    \hline
$ det[G_1F_2\cdots F_{n-1}F_nG_{n+1}F_{n+2}] $& $L_{n+2}\,[-(n-2a-an)n(n+1)] $&$L_{n+2}\,[-(n-2a-an)n(n+1)] $\\
    \hline
$ det[G_1F_2\cdots F_{n-1}F_nG_{n+1}G_{n+2}] $& $L_{n+2}\,[\frac{e^{i\theta}}{4}(n-2a-an)^2n^2] $&$L_{n+2}\,[-\frac{e^{i\theta}}{4}(n-2a-an)^2n^2] $\\
    \hline
$ det[G_1F_2\cdots F_{n-1}G_nF_{n+1}F_{n+2}] $& $L_{n+2}\,[-(n-2-an)(n+2)] $&$L_{n+2}\,[-(n-2-an)(n+2)] $\\
    \hline
$ det[G_1F_2\cdots F_{n-1}G_nG_{n+1}F_{n+2}] $& $L_{n+2}\,[-\frac{e^{i\theta}}{4}(n-2a-an)(n-2-an)n(n+2)] $&$L_{n+2}\,[\frac{e^{i\theta}}{4}(n-2a-an)(n-2-an)n(n+2)] $\\
    \hline
$ det[G_1F_2\cdots G_{n-1}F_nF_{n+1}F_{n+2}] $& $L_{n+2}\,[(n-2-an)(n-2)(n+1)] $&$L_{n+2}\,[(n-2-an)(n-2)(n+1)] $\\
    \hline
$ det[G_1F_2\cdots G_{n-1}F_nF_{n+1}G_{n+2}] $& $L_{n+2}\,[-\frac{e^{i\theta}}{4}(n-2a-an)(n-2-an)n(n-2)] $&$L_{n+2}\,[\frac{e^{i\theta}}{4}(n-2a-an)(n-2-an)n(n-2)] $\\
    \hline
$ det[G_1F_2\cdots G_{n-1}G_nF_{n+1}F_{n+2}] $& $L_{n+2}\,[\frac{e^{i\theta}}{4}(n-2-an)^2(n+2)(n-2)] $&$L_{n+2}\,[-\frac{e^{i\theta}}{4}(n-2-an)^2(n+2)(n-2)] $\\
    \hline
$ det[G_1G_2\cdots F_{n-1}F_nF_{n+1}F_{n+2}] $& $L_{n+2}\,[4e^{-2i\theta}] $&$L_{n+2}\,[4e^{-2i\theta}] $\\
    \hline
$ det[G_1G_2\cdots F_{n-1}F_nF_{n+1}G_{n+2}] $& $L_{n+2}\,[e^{-i\theta}(n-2a-an)n] $&$L_{n+2}\,[-e^{-i\theta}(n-2a-an)n] $\\
    \hline
$ det[G_1G_2\cdots F_{n-1}F_nG_{n+1}F_{n+2}] $& $L_{n+2}\,[e^{-i\theta}(n-2a-an)n] $&$L_{n+2}\,[-e^{-i\theta}(n-2a-an)n] $\\
    \hline
$ det[G_1G_2\cdots F_{n-1}F_nG_{n+1}G_{n+2}] $& $L_{n+2}\,[\frac{1}{4}(n-2a-an)^2n^2] $&$L_{n+2}\,[\frac{1}{4}(n-2a-an)^2n^2] $\\
    \hline
$ det[G_1G_2\cdots F_{n-1}G_nF_{n+1}F_{n+2}] $& $L_{n+2}\,[-e^{-i\theta}(n-2-an)(n-2)] $&$L_{n+2}\,[e^{-i\theta}(n-2-an)(n-2)] $\\
    \hline
$ det[G_1G_2\cdots F_{n-1}G_nG_{n+1}F_{n+2}] $& $L_{n+2}\,[-\frac{1}{4}(n-2a-an)(n-2-an)n(n-2)] $&$L_{n+2}\,[-\frac{1}{4}(n-2a-an)(n-2-an)n(n-2)] $\\
    \hline
$ det[G_1G_2\cdots G_{n-1}F_nF_{n+1}F_{n+2}] $& $L_{n+2}\,[-e^{-i\theta}(n-2-an)(n-2)] $&$L_{n+2}\,[e^{-i\theta}(n-2-an)(n-2)] $\\
    \hline
$ det[G_1G_2\cdots G_{n-1}F_nF_{n+1}G_{n+2}] $& $L_{n+2}\,[-\frac{1}{4}(n-2a-an)(n-2-an)n(n-2)] $&$L_{n+2}\,[-\frac{1}{4}(n-2a-an)(n-2-an)n(n-2)] $\\
    \hline
$ det[G_1G_2\cdots G_{n-1}G_nF_{n+1}F_{n+2}] $& $L_{n+2}\,[\frac{1}{4}(n-2-an)^2(n-2)^2] $&$L_{n+2}\,[\frac{1}{4}(n-2-an)(n-2)^2] $\\
    \hline
$ det[G_1H_2\cdots F_{n-1}F_nF_{n+1}F_{n+2}] $& $L_{n+2}\,[4e^{-i\theta}an] $&$L_{n+2}\,[-4e^{-i\theta}an] $\\
    \hline
$ det[G_1H_2\cdots F_{n-1}F_nG_{n+1}F_{n+2}] $& $L_{n+2}\,[a(n-2a-an)n^2] $&$L_{n+2}\,[a(n-2a-an)n^2] $\\
    \hline
$ det[G_1H_2\cdots G_{n-1}F_nF_{n+1}F_{n+2}] $& $L_{n+2}\,[-a(n-2-an)n(n-2)] $&$L_{n+2}\,[-a(n-2-an)n(n-2)] $\\
    \hline
$ det[H_1F_2\cdots F_{n-1}F_nF_{n+1}F_{n+2}] $& $L_{n+2}\,[-2an(2n+2)] $&$L_{n+2}\,[-2an(n+2)] $\\
    \hline
$ det[H_1F_2\cdots F_{n-1}F_nF_{n+1}G_{n+2}] $& $L_{n+2}\,[e^{i\theta}an^2(n-2a-an)] $&$L_{n+2}\,[-e^{i\theta}an^2(n-2a-an)] $\\
    \hline
$ det[H_1F_2\cdots F_{n-1}G_nF_{n+1}F_{n+2}] $& $L_{n+2}\,[-e^{i\theta}an(n+2)(n-2-an)] $&$L_{n+2}\,[e^{i\theta}an(n+2)(n-2-an)] $\\
    \hline
$ det[H_1G_2\cdots F_{n-1}F_nF_{n+1}F_{n+2}] $& $L_{n+2}\,[4e^{-i\theta}an] $&$L_{n+2}\,[-4e^{-i\theta}an] $\\
    \hline
$ det[H_1G_2\cdots F_{n-1}F_nF_{n+1}G_{n+2}] $& $L_{n+2}\,[an^2(n-2a-an)] $&$L_{n+2}\,[an^2(n-2a-an)] $\\
    \hline
$ det[H_1G_2\cdots F_{n-1}G_nF_{n+1}F_{n+2}] $& $L_{n+2}\,[-an(n-2)(n-2-an)] $&$L_{n+2}\,[-an(n-2)(n-2-an)] $\\
    \hline
$ det[H_1H_2\cdots F_{n-1}F_nF_{n+1}F_{n+2}] $& $L_{n+2}\,[4a^2n^2] $&$L_{n+2}\,[4a^2n^2] $\\
    \hline
\end{tabular}
\vspace{3cm}

\noindent{\emph{Acknowledgement: The first author is thankful to the Council of Scientific and Industrial Research, New Delhi, for financial support vide grant no. 09/797/0006/2010 EMR-1.}}

{

\end{document}